\documentclass[a4paper,11pt]{article}

\usepackage{graphicx}
\usepackage{verbatim}
\usepackage{fancybox}
\usepackage{textcomp}
\usepackage{amsmath}
\usepackage{amsfonts}
\usepackage{amssymb}
\usepackage{color}  
\usepackage{amsthm}
\usepackage{mathrsfs}  
\usepackage{eurosym} 
\usepackage[utf8]{inputenc}
\usepackage[T1]{fontenc}
\usepackage{makeidx,bbm}
\usepackage[textwidth=100pt]{todonotes}
\usepackage{geometry}
 \usepackage[unicode=true,bookmarks=false,breaklinks=false,pdfborder={0 0 1},backref=false,colorlinks=false]{hyperref}

\newtheorem{theorem}{Theorem}[section]
\newtheorem{lemma}[theorem]{Lemma}

\newtheorem{rem}[theorem]{Remark}

\newcommand{\R}{\mathbb{R}}
\newcommand{\Z}{\mathbb{Z}}
\newcommand{\E}{\mathbb{E}}

\newcommand{\mum}{{\boldsymbol\mu}}
\newcommand{\Zm}{{\mathbf Z}}
\newcommand{\calH}{\mathcal{H}}

\newcommand{\equa}[1]{\begin{eqnarray} #1 \end{eqnarray}}
\newcommand{\ind}[1]{\mathbbm{1}_{\left[ {#1} \right] }}

\newcommand{\ee}{\mathbf e}
\newcommand{\bfa}{\mathbf a}
\newcommand{\bfq}{\mathbf q}

\title{A second note on the discrete Gaussian Free Field \\ with
  disordered pinning on $\Z^d$, $d\geq2$}
\author{Loren Coquille, Piotr Mi\l{}o\'s}
\date{\today}

\begin{document}
\maketitle

\begin{abstract}
We study the discrete massless Gaussian Free Field on $\Z^d$, $d\geq2$,
in the presence of a disordered square-well potential supported on a finite strip
around zero. The disorder is introduced by reward/penalty
interaction coefficients, which are given by i.i.d. random
variables. 

In the previous note \cite{CoqMil2013_1}, we proved under
minimal assumptions on the law of the environment, that the quenched free energy associated to this model
exists in $\R^+$, is deterministic, and strictly smaller than the
annealed free energy whenever the latter is strictly
positive. 

Here we consider Bernoulli reward/penalty coefficients~$b\cdot e_x+h$ with $e_x\sim
\text{Bernoulli}_{1/2}(-1,+1)$ for all $x\in\Z^d$, and $b>0$,
$h\in\R$. We prove that in the plane
$(b,h)$, the quenched critical line
(separating the phases of positive and zero free energy) lies strictly
below the line $h=0$, showing in particular that there exists a
non trivial region where the field is localized though repulsed on
average by the environment. 

\end{abstract}
\strut

  

\section{The model}  

We study the discrete Gaussian Free Field with a disordered
square-well potential. For $\Lambda$ a finite subset of $\Z^d$, denoted by $\Lambda\Subset\Z^d$, let $\varphi=(\varphi_x)_{x\in\Lambda}$
represent the heights over sites of
$\Lambda$. The values of $\varphi_x$ can also be seen as continuous
unbounded (spin) variables, we will refer to $\varphi$ as ``the
interface'' or ``the field''.
 
Let $\Omega=\R^{\Z^d}$ be the set of configurations. The finite volume Gibbs measure in $\Lambda$ for the discrete Gaussian Free Field with disordered square-well potential,
and $0$ boundary conditions, is the probability measure on $\Omega$
defined by :  
\equa{\label{astrip_measure}
\mu_{\Lambda}^{\ee,0}(d\varphi)=\frac1{Z_{\Lambda}^{\ee,0}}\exp\left({-\beta
    \calH_\Lambda(\varphi)+\beta\sum_{x\in\Lambda}
 (b\cdot e_x +h) \ind{\varphi_x\in[-a,a]}}\right)\prod_{x\in\Lambda}d\varphi_x\prod_{y\in\Lambda^c}\delta_0(d\varphi_y).}
where $a,\beta,b >0 $, $h\in\R$ and $\calH_\Lambda(\varphi)$ is given by 
\begin{equation}
	\calH_\Lambda(\varphi)=\frac1{4d}\sum_{\substack{\{x,y\}\cap\Lambda\neq\varnothing
	    \\ x\sim y}}(\varphi_x-\varphi_y)^2,\label{eq:gffHamiltonian}
\end{equation}
where $x\sim y$ denotes an edge of the graph $\Z^d$ and $\ind{A}$ denotes the indicator function of $A$. An environment is denoted as
$\ee:=(e_x)_{x\in\Lambda}$.
We consider here $\ee$ given by i.i.d.\ random
variables $$e_x\sim\text{Bernoulli}_{1/2}(-1,+1).$$ 
The parameter $b$ is usually called the ``intensity of the disorder'',
while $h$ is its average. 
The disordered potential attracts or repulses the field at heights belonging to $[-a,a]$. 
$Z_{\Lambda}^{\ee,0}$ is the partition function, i.e.\ it normalizes
$\mu_{\Lambda}^{\ee,0}$ so it is a probability measure. The
superscipt $0$ reminds the boundary condition, it is added to the
notation compared to \cite{CoqMil2013_1} because it will be useful below.
We stress that our model contains two levels of randomness. The first one is $\ee$ which we refer to as ``the environment''. The second one is the actual interface model whose low depends on the realization of $\ee$. 

The inverse temperature parameter $\beta$ enters only in a trivial way. Indeed, if we replace
the field $(\varphi_x)_{x\in\Lambda}$ by
$(\sqrt{\beta}\phi_x)_{x\in\Lambda}$, $a$ by $\sqrt{\beta}a$, and
$(b\cdot e_x+h)_{x\in\Lambda}$ by
$({\beta}(b\cdot e_x+h))_{x\in\Lambda}$ we have transformed the model to
temperature parameter $\beta=1$. In the sequel, we will therefore work
with $\beta=1$.

The dimensions 1 and 2 are physically relevant as interface
models. In this paper we focus on $d\geq2$ since 1-dimensional
models have been well-studied in the last decade
(see \cite{CoqMil2013_1} for a historical introduction).

The questions we are addressing in this framework are the usual ones
concerning statistical mechanics models in random environment : Is the quenched
free energy non-random ? Does it differ from the annealed one ? Can we
give a physical meaning to the strict positivity (resp. vanishing) of
the free energy ? What can be said concerning the
quenched and annealed critical lines (surfaces) in the space of the relevant
parameters of the system ? 

In the previous note \cite{CoqMil2013_1}, we proved under
minimal assumptions on the law of the environment, that the quenched free energy associated to this model
exists in $\R^+$, is deterministic, and strictly smaller than the
annealed free energy whenever the latter is strictly
positive. 

Here we investigate the phase diagram of the model : in the
  plane $(b,h)$, we prove that the quenched critical line (separating the
phases of positive and zero free energy) lies strictly below the line
$h=0$. Thus there exists a non trivial region where the
field is localized though repulsed on average by the
environment.

\section{Results}

We define the quenched (resp. annealed) free
energy per site in $\Lambda \Subset \Z^d$ by :
\begin{equation}
	f^\bfq_{\Lambda} (\ee)=|\Lambda|^{-1}\log \left(\frac{Z_{\Lambda}^{\ee,0}}{Z_\Lambda^{0,0}}\right) ,\quad
	f_{{\Lambda}}^\bfa(\ee)=|\Lambda|^{-1}\log\left(\frac{\E Z_{\Lambda}^{\ee,0}}{Z_\Lambda^{0,0}}\right), \label{eq:freeEnergyDef}
\end{equation}
where $Z_\Lambda^{0,0}$ denotes the partition function of the model with
no potential, $e_x \equiv 0$ (i.e. of the Gaussian free field). 
In the case when $\Lambda = \Lambda_n=\lbrace 0, ..., n-1\rbrace^d$
we will use short forms $f^\bfq_{n} (\ee)$ and $f^\bfa_{n} (\ee)$. By
the Jensen inequality, we have $f^\bfq(\ee)\leq
f^\bfa(\ee)$. Moreover, it
is not difficult to see that the annealed model corresponds to the
model with constant (we will also say homogenous) pinning with the strength
\begin{equation}\ell(\ee):=\log(\E(e^{b\cdot e_x+h})).\label{eq:effectiveAnnealedPotential}\end{equation}
for all $x\in\Lambda$. In other words $ \E
Z^{\ee,0}_{\Lambda}=Z^{\ell(\ee),0}_{\Lambda}$. 

In \cite[Fact 2.3]{CoqMil2013_1} we proved that for any environment 
$\ee$ such that the annealed model exists, i.e.\ $\E(e^{b\cdot e_x+h})<\infty$, both the quenched and annealed free energies are
non-negative. 
This motivates the following notions. We introduce the
quenched (resp. annealed) critical lines, which are delimiting the
region where $f^\bfq(\ee)=0$ (resp. $f^{\bfa} (\ee)=0$) from the region
$f^\bfq(\ee)>0$ (resp. $f^{\bfa} (\ee)>0$). 
$$ h_c^\bfq(b):=\sup\lbrace h\in\R :
f^\bfq(\ee) =0\rbrace\quad \mbox{ and }\quad 
{h}_c^{\bfa}(b):=\sup\lbrace h\in\R :
f^{\bfa}(\ee)=0\rbrace$$

We are interested in describing the
behavior of these quantities in the phase diagram described by the
plane $(b,h)$.
\noindent Knowing the behavior of the
homogenous model for positive pinning \cite{BolVel2001}, we easily deduce that the annealed critical line is
given by the equation $ \ell(\ee)=0$.\\
Note that $f^\bfq(\ee)\leq f^{\bfa} (\ee)$ implies that
$h_c^\bfq(b)\geq{h}_c^\bfa(b).$ 
In Theorem \ref{thm:shift_quenched} we show that the quenched critical
line lies strictly below the axis $h=0$ in the neighborhood of
$b=0$ for all $d\geq 2$. 
Our
result shows in particular that there exists a non
trivial region where $h<0$, $b>0$ and $f^\bfq(\ee)>0$, i.e.\ where the
field is localized though it is repulsed on average by the
environment.
 
Note that we don't have any estimate on the
behavior of
$h_c^\bfq(b)-{h}_c^\bfa(b)$.

\begin{theorem}  
\label{thm:shift_quenched}
Let $\ee\sim\otimes_{x\in\Z^d}\text{Bernoulli}_{1/2}(-1,+1)$. Then,\\

\noindent For $d\geq2$, the quenched critical line is
located in the quadrant $\lbrace (b,h) : b\geq0,
h<0 \rbrace$.\\
  
\noindent More precisely, there exists some $C, C'>0$
depending on $d,a$ only and
$\epsilon\in(0,1)$ such that
for any environment $\ee$ which fulfills $b+h>0$, $-\epsilon<-b+h<0$ and 
$$ \begin{cases}
h>\frac{C'(-b+h)^2}{\log(b-h)} & \mbox{ for } d=2\\
h>- C\cdot(-b+h)^2 &\mbox{ for } d\geq3,
\end{cases}$$
we have $f^\bfq(\ee)>0$.
\end{theorem}

\begin{rem}
\begin{enumerate}
\item A sketch of these bounds in the plane
  ($b,h$) can be seen on
  Figure~\ref{fig:critical_curves}. Moreover, the bound for $d\geq3$ can be rewritten
  as $h>-C''(d,a)\cdot b^2$.  
\item Jensen's inequality gives us an upper bound on
$C, C'$. Indeed, as $f^\bfa (\ee)\geq f^\bfq(\ee)$, if $f^\bfa (\ee)=0$ then
$f^\bfq(\ee)=0$. In particular, we must have $-C\leq
\frac{\partial^2}{\partial b^2}{h_c\vert}_{b=0}<0$. Our
result gives thus an upper-bound on the behavior of the quenched
critical line near $b=0$.
\end{enumerate}
\end{rem}
    
\begin{figure}[h!]
    \centering
    \includegraphics[width=6.5cm]{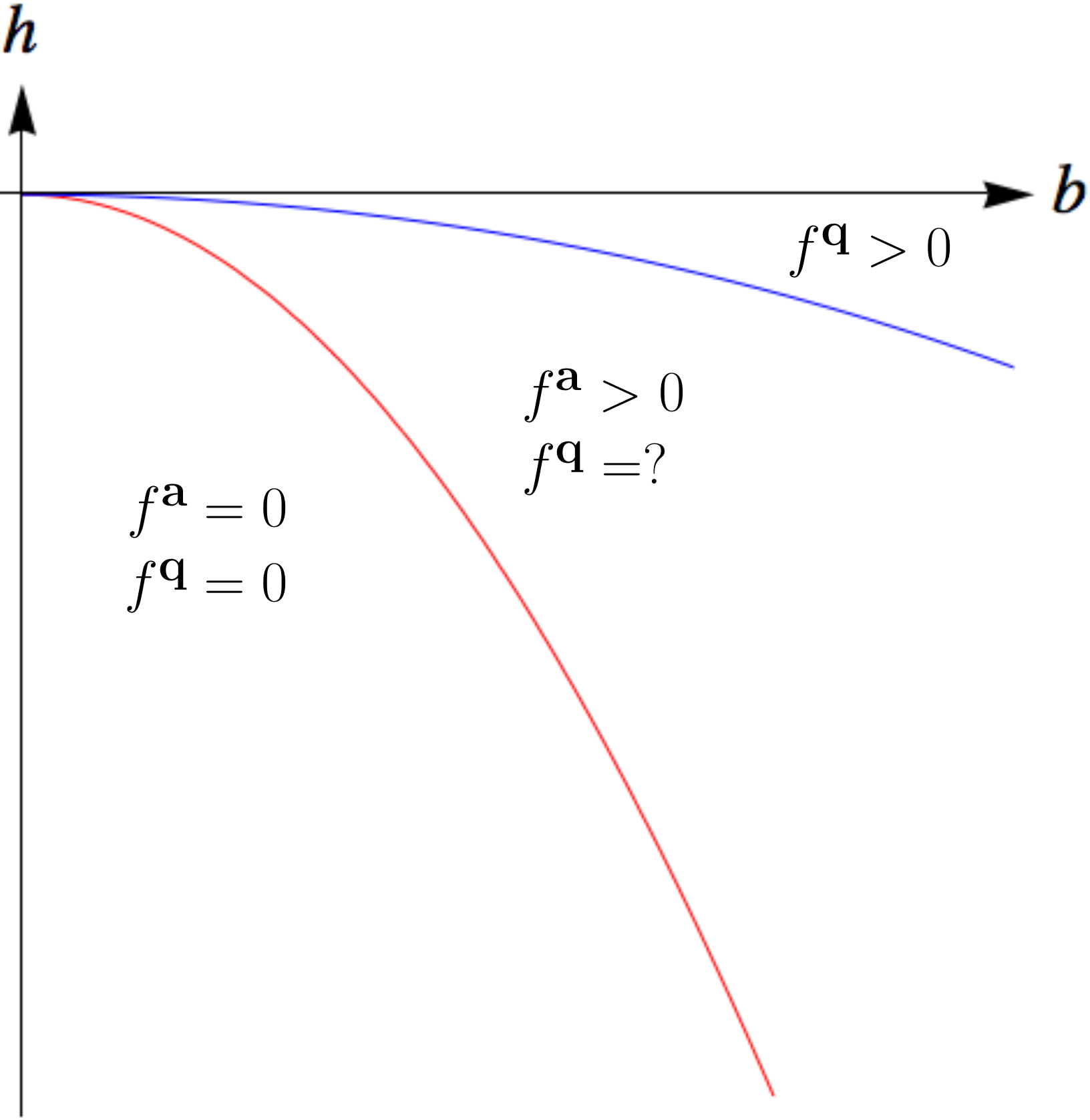}
    \caption{(Color online) {Phase diagram of the model. The red curve is the
        annealed critical line; the blue one is our bound on the quenched critical line.\label{fig:critical_curves}    
      }}
\end{figure}

\subsection{Related results}

The same type of result has been proven for (1$d$) polymer models in
great generality in \cite{AleSid2006}, where Alexander and Sidoravicius consider a
polymer, with monomer locations modeled by the trajectory of a Markov
chain $(X_i)_{i\in\Z}$, in the presence of a potential {(usually called a ``defect line'')} that interacts with the polymer
when it visits 0. Formally, the model is given by weighting the realization of the chain with the Boltzmann term
$$ \exp\left(\beta\sum_{i=1}^n(u+V_i)\ind{X_i=0}\right).$$
with $(V_i)_{i\in\Z}$ an i.i.d.\ sequence of $0$-mean random variables.
They studied the localization transition in this model.
We say that the polymer is pinned, if a positive fraction of monomers is at 0. {In the plane $(\beta, u)$
critical lines are defined as above: for $\beta$ fixed, $u_c^\bfq(\beta)$
(resp. $u_c^{\bfa}(\beta)$) is the value of $u$ above which the polymer is pinned
with probability 1
(for the quenched (resp. annealed) measure)}. They showed that the quenched free
energy and critical point are non-random, calculated the critical point for a deterministic
interaction (i.e.\ $V_i \equiv 0$) and proved 
that the critical point in the quenched case is strictly smaller. 

When the underlying chain is a symmetric simple random walk
on $\Z$, the deterministic critical point is $0$, so having the quenched critical point $u_c(\beta)$ strictly
negative means that, even when the disorder is repulsive on average,
the chain is pinned. This result was obtained by Galluccio and Graber
in \cite{GalGra1996} for a periodic potential, which is frequently
used in the physics literature as a ``toy model'' for random
environments.

A much shorter proof with explicit bounds can be found in
\cite{Gia2007}, in less generality, but \cite{Hol2009} contains a revisited proof with explicit estimates and
weakening the assumptions on the underlying model.

Note that for polymers, or discrete height interfaces, one need a
coarse graining procedure to achieve the proof.
In our case, as we will see in the next section, we can shift the continuous interface where the
environment is unfavorable, and this has a small cost in dimension
$d\geq3$. The procedure is a bit more complicated in dimension 2 and
we have to localize the field before by introducing a small mass.

\subsection{Proof of Theorem \ref{thm:shift_quenched}}
\subsubsection{Case $d\geq3$}
We assume $-b+h<0<b+h$ i.e.\ the environment is repulsive if $e_x=-1$
while it is attractive if $e_x=+1$.

The idea is to tilt the measure such that the field $\varphi$ is
shifted up of an amount $s$ on the sites $x$ for which $b\cdot e_x+h<0$. In
this way the shift of the field follows the
environment. For some technical
reasons, we need to work with the measure with boundary condition $a$,
so perform two changes of measure (first one changing boundary
condition and the second one to following the environment).
Let $s>0$ (to be fixed later).
\[
f^\bfq_{\Lambda_n}(\ee)=n^{-d}\log\mu^{0,a}_{\Lambda_n,\ee,s}\left(\frac{\mathrm{d}{\mu_{\Lambda_n}^{0,a}}}{\mathrm{d}{\mu^{0,a}_{\Lambda_n,\ee,s}}}\frac{\mathrm{d}{\mu_{\Lambda_n}^{0,0}}}{\mathrm{d}{\mu^{0,a}_{\Lambda_n}}}\exp\left(\sum_{x\in\Lambda_n}(b\cdot e_x+h)\ind{\varphi_x\in[-a,a]}\right)\right),
\]
where $\left({\varphi_{x}}\right)_{x\in\Lambda_n}$ under $\mu^{0,a}_{\Lambda_n,\ee,s}$
is distributed as $\left({\varphi_{x}+s\ind{{(b\cdot e_x+h)<0}}}\right)_{x\in\Lambda_n}$
under $\mu^{0,a}_{\Lambda_n}$. More formally, introducing 
$T_{\ee,s}:
(\left({\varphi_{x}}\right)_{x\in\Lambda_n})\mapsto\left({\varphi_{x}+s\ind{{(b\cdot e_x+h)<0}}}\right)_{x\in\Lambda_n}$,
we define $\mu^{0,a}_{\Lambda_n,\ee,s}$ as $\mu^{0,a}_{\Lambda_n}\circ
T_{\ee,s}^{-1}$. \\
Using Jensen's inequality, we get
\begin{eqnarray*}
f_{\Lambda_n}^\bfq(\ee)&=&n^{-d}\log\left[\mu^{0,a}_{\Lambda_n,\ee,s}\exp\left(\sum_{x\in\Lambda_n}(b\cdot e_x+h)\ind{\varphi_x\in[-a,a]}
+\log\frac{\mathrm
      d{\mu^{0,a}_{\Lambda_n}}}{\mathrm d{\mu^{0,a}_{\Lambda_n,\ee,s}}}
+\log\frac{\mathrm
      d{\mu ^{0,0}_{\Lambda_n}}}{\mathrm d{\mu^{0,a}_{\Lambda_n}}}\right)\right]\\
&\geq&
n^{-d}\mu^{0,a}_{\Lambda_n,\ee,s}\left(\sum_{x\in\Lambda_n}(b\cdot e_x+h)\ind{\varphi_x\in[-a,a]}
+\underbrace{\log\frac{\mathrm
      d{\mu^{0,a}_{\Lambda_n}}}{\mathrm d{\mu^{0,a}_{\Lambda_n,\ee,s}}}}_{(1)}
+\underbrace{\log\frac{\mathrm
      d{\mu ^{0,0}_{\Lambda_n}}}{\mathrm d{\mu^{0,a}_{\Lambda_n}}}}_{(2)}\right)\\
\end{eqnarray*}
As $Z ^{0,0}_{\Lambda_n,\ee,s}=Z ^{0,0}_{\Lambda_n}$ (which follows by change of variables in the Gaussian
integral), the first term can be written as
\[
(1)=-\frac1{4d}\sum_{\substack{\{x,y\}\cap\Lambda_n\neq\varnothing\\x\sim
  y}}(\varphi_x-\varphi_y)^2-(\hat\varphi_x-\hat\varphi_y)^2
\]
where $\hat\varphi_x:=\varphi_x+s\ind{b\cdot e_x+h<0}$. Hence, using the
definition of $\mu^{0,a}_{\Lambda_n,\ee,s}$,
\begin{eqnarray*}
n^{-d} \mu^{0,a}_{\Lambda_n,\ee,s}((1))
&=&-\frac{n^{-d}}{4d} \mu^{0,a}_{\Lambda_n}\left(\sum_{\substack{\{x,y\}\cap\Lambda_n\neq\varnothing\\x\sim
  y}}(\hat\varphi_x-\hat\varphi_y)^2-(\varphi_x-\varphi_y)^2\right)\\
&=& -\frac{s^2n^{-d}}{4d} \sum_{\substack{\{x,y\}\cap\Lambda_n\neq\varnothing\\x\sim
  y}} (\ind{b\cdot e_x+h<0}-\ind{b\cdot e_y+h<0})^2
\end{eqnarray*}
The second term contains only boundary contribution of order
$n^{d-1}$. Indeed,
\[
(2)=\left(2a\sum_{x\in\partial\Lambda_n}\varphi_x-a^2\vert\partial\Lambda_n\vert\right)+\log\left(\frac{Z^{0,a}_{\Lambda_n}}{Z ^{0,0}_{\Lambda_n}}\right)
\geq
\left(2a\sum_{x\in\partial\Lambda_n}\varphi_x-a^2\vert\partial\Lambda_n\vert\right)-C n^{d-1}
\] 
Hence,
\[
n^{-d}\mu^{0,a}_{\Lambda_n,\ee,s}((2))\geq
{2a}\cdot{n^{-d}}\sum_{x\in\partial\Lambda_n}\mu^{0,a}_{\Lambda_n} (\hat\varphi_x)-C
n^{-1} 
\geq  s\sum_{x\in\partial\Lambda_n}\ind{b\cdot e_x+h<0}-Cn^{-1}
\geq -Cn^{-1}
\] 
We get
\begin{eqnarray*}
f_{\Lambda_n}^\bfq(\ee)&\geq&
n^{-d}\sum_{x\in\Lambda_n}(b\cdot e_x+h)\mu^{0,a}_{{\Lambda_n}}(\hat\varphi_x\in[-a,a])
-\frac{s^2n^{-d}}{4d} \sum_{\substack{\{x,y\}\cap\Lambda_n\neq\varnothing\\x\sim
  y}} (\ind{b\cdot e_x+h<0}-\ind{b\cdot e_y+h<0})^2-Cn^{-1}
\end{eqnarray*}  
Now we use the fact that the marginal laws of all $\varphi_x$, $x\in\Lambda_n$ under
$\mu^{0,a}_{{\Lambda_n}}$ are Gaussian variables centered at $a$, i.e.\ $\varphi_{x}\sim\mathcal{N}(a,\sigma_n^x)$ where
$\sigma_n^x=\mbox{Var}^{0,a}_{\Lambda_n} (\varphi_x)\leq
\mbox{Var}^{0,a}_\infty(\varphi_x) \leq c(d)<\infty$ for $d\geq 3$. Therefore,
\begin{eqnarray}\label{overlap_gaussian}
\mu^{0,a}_{{\Lambda_n}}({\varphi_{x}\in[-a,a]})-\mu^{0,a}_{{\Lambda_n}}({\varphi_{x}+s\in[-a,a]})
&=& \mu ^{0,0}_{{\Lambda_n}}({\varphi_{x}\in[-2a,0]})-\mu ^{0,0}_{{\Lambda_n}}({\varphi_{x}\in[-2a-s,-s]})\nonumber\\
&=&C\left( \int_{-s}^0-\int_{-2a-s}^{-2a} \right) e^{-y^2/2{\sigma_n^x}^2}
dy\nonumber \\
&\asymp&s\quad \mbox{ as } n\to\infty,
\end{eqnarray}
for $c(d)>>s$. In particular we will
use that :
\begin{eqnarray*}
\mu^{0,a}_{{\Lambda_n}}({\varphi_{x}\in[-a,a]})-\mu^{0,a}_{{\Lambda_n}}({\varphi_{x}+s\in[-a,a]})
&\geq&C_1(d,a)\cdot s,
\end{eqnarray*}
for some $C_1(d,a)>0$. 
\begin{align*}
f_n^\bfq(\ee)\geq
n^{-d}\sum_{x\in\Lambda_n}&(b\cdot e_x+h)\left(\mu^{0,a}_{\Lambda_n} (\varphi_x\in[-a,a])-C_1(d,a)s\ind{b\cdot e_x+h<0}\right)\\
&-\frac{s^2n^{-d}}{4d} \sum_{\substack{\{x,y\}\cap\Lambda_n\neq\varnothing\\x\sim
  y}} (\ind{b\cdot e_x+h<0}-\ind{b\cdot e_y+h<0})^2-Cn^{-1}
\end{align*}
Observe that
$\mu^{0,a}_{\Lambda_n} (\varphi_x\in[-a,a])=\mu ^{0,0}_{\Lambda_n} (\varphi_x\in[-2a,0])\geq\mu ^{0,0}_\infty(\varphi_x\in[-2a,0])\geq
C_2(d,a)$ for some $C_2(d,a)>0$.

By taking the
expectation with respect to the environment, using the bounded
convergence theorem and the fact that $f^\bfq(\ee)=\E(f^\bfq(\ee))$
(cf.\cite[Theorem 2.1]{CoqMil2013_1}) we get :
{
\begin{eqnarray}\label{parabola}
f^\bfq(\ee)=\lim_{n\to\infty}\E f^\bfq_{\Lambda_n}(\ee)&\geq&
h C_2(d,a)-\frac{sC_1(d,a)}{2}(-b+h)-\frac{s^2}{16} 
\end{eqnarray}
We may optimize over $s$ as the left hand side does not depend on it. Doing this one checks that $f^\bfq(\ee)> 0$ as soon as 
\[
h>-\frac{C_1(d,a)}{C_2(d,a)}\cdot (-b+h)^2=:-K(d,a)\cdot (-b+h)^2
\]    
This gives the implicit equation in terms of the variance $b^2$ of $b\cdot e_x+h$ :
\begin{eqnarray*}  
h>b-\frac1{2K}+\frac12\sqrt{\frac1{K^2}-\frac{8b}{K}} 
= -K b^2+O(b^3)
\end{eqnarray*}
\noindent The annealed critical curve as well as this bound are drawn on Figure \ref{fig:critical_curves}.
\noindent We recall that \eqref{overlap_gaussian} is valid under
assumption that $s$ is small. The maximum of (\ref{parabola}) is
realized at $s_{\max}=-4C_1\cdot (-b+h)$, thus it is enough to assume that $(-b+h)$ is small.
} 
\qed

\subsubsection{Case $d=2$}
In the case $d=2$, the variance of the Gaussian free field diverges
with the size of the box, so we cannot use the previous estimates. To circumvent this problem we introduce the so-called massive free field. Let $m>0$,
	\begin{equation}
	\mum^{0,\zeta}_{{\Lambda_n},m}(d{\varphi})=\frac{1}{\Zm_{{\Lambda_n},m}^{0,\zeta}}\exp\left(
-\calH_{\Lambda_n}(\varphi)
- m^2\sum_{x\in \Lambda_n }(\varphi_x-\zeta)^2\right)\prod_{x\in\Lambda_n}d{\varphi_{x}}\prod_{x\in\partial\Lambda_n}\delta_{0}(d{\varphi_{x}}),
	\end{equation}
where $\calH_\Lambda(\varphi)$ is defined in \eqref{eq:gffHamiltonian}.
Known facts about this model can be found in 
\cite[Section 3.3]{DemFun2005}. In particular, the random walk
representation for the massless GFF \cite[(1.3)]{BolVel2001} is still true, but for a random walk $Y_t$
that is killed with rate $\xi(m)=\frac{m^2}{1+m^2}$, namely at each
time $\ell$, if the walk has not already been killed, it is killed
with probability $\xi(m)$, where the killing is independent of the
walk. We write its law $P_x$ when it starts at $x$.
\newpage
\begin{lemma}\label{lem:massive}  Let $d=2$. Then,
\begin{enumerate}    
\item There exists some $C_1>0$  such that for $n$ large enough,
  $m>0$ small enough and
all $x\in\Lambda_n$, 
$$\mum^{0,0}_{{\Lambda_n},m}(\varphi^2_x)\leq
C_1|\log(m)|.$$
\item There exists some $C_2>0$ such that for $n$ large enough and $m>0$
  small enough, we have
\[
	n^{-2}\log\frac{\Zm^{0,0}_{{\Lambda_n},m}}{Z^{0,0}_{{\Lambda_n}}} \geq -C_2m^2|\log(m)|.
\]
\end{enumerate}
\end{lemma}  
\begin{proof}
These bounds are rather standard. We give here the main
steps of the proofs with some references. For the first claim, we use
the random walk representation \cite{DemFun2005} to write
\[   
\mum^{0,0}_{{\Lambda_n},m}(\varphi_x^2)=\sum_{\ell=0}^\infty P_x(Y_\ell=x\,,\,
\tau_{\Lambda_n}\wedge\aleph>\ell)=
\sum_{\ell=0}^\infty (1-\xi(m))^\ell P_x(Y_\ell=x\,,\,
\tau_{\Lambda_n}>\ell)   
\]
where $\tau_{\Lambda_n}$ is the first exit time of $\Lambda_n$ and
$\aleph$ is the killing time of the random walk
$Y_t$. Hence,
\begin{eqnarray}\label{eq:massive_variance}
\mum^{0,0}_{{\Lambda_n},m}(\varphi_x^2)\leq \mum^{0,0}_{{\Lambda_n},m}(\varphi_0^2)\leq
\sum_{\ell=1}^\infty (1-\xi(m))^\ell P_0(X_\ell=0)
\end{eqnarray}
where $X_\ell$ is a simple random walk (without killing). The projections of $X_\ell$ onto the two coordinate axis are two
independent $1-$dimensional random walks $X_\ell^1$ and $X_\ell^2$, then by Stirling formula,
\[
P_0(X_{2 \ell}=0)=(P_0(X_{2 \ell}^1=0))^2=\left(\binom{2 \ell}{\ell}2^{-2 \ell}\right)^2=\frac1{\pi
  \ell}(1+o(1))\quad\mbox{ as }\ell\to\infty
\]
The asymptotics of (\ref{eq:massive_variance}) for small $m$ gives the
desired upper-bound.\\
To prove the second claim we use the representation of the partition
function described in 
\cite[p.542]{BolIof1997} (it applies to the massive GFF with an
obvious modification). We denote by $\tilde P$ the coupling of a random
walk $X_n$ and a killed random walk $Y_n$ such that $Y_n=X_n$ up to
its killing time $\aleph$, 
	\begin{eqnarray}\label{eq:ratio_Z}
	|\Lambda_n|^{-1}\log\frac{Z_{{\Lambda_n}}^{0,0}}{\Zm_{{\Lambda_n},m}^{0,0}}
        &=&|\Lambda_n|^{-1}\left(\frac12\sum_{x\in\Lambda_n}\sum_{\ell=1}^{\infty}\frac{1}{2
            \ell}\left(\tilde P_x (X_{2 \ell}=x,\tau_{\Lambda_n}>2 \ell)-\tilde P_x (Y_{2 \ell}=x,\tau_{\Lambda_n}\wedge\aleph>2 \ell)\right) \right) \nonumber\\
	 	&\leq& \frac12\sum_{\ell=1}^{\infty}\frac{1}{2 \ell}\left(\tilde P_0(X_{2 \ell}=0,\tau_{\Lambda_n}>2 \ell)-\tilde P_0(Y_{2 \ell}=0,\tau_{\Lambda_n}\wedge\aleph>2 \ell)\right) \nonumber\\
	 	&=& \frac12\sum_{\ell=1}^{\infty}\frac{1}{2
                  \ell}\tilde P_0(X_{2
                    \ell}=0,\tau_{\Lambda_n}>2 \ell,\aleph\leq 2 \ell)\nonumber\\
&\leq&\frac12\sum_{\ell=1}^{\infty}\frac{1}{2 \ell}\tilde
                P_0(X_{2 \ell}=0) \left(1-(1-\xi(m))^{2\ell}\right)
	\end{eqnarray}	
Using the same estimate as in
(\ref{eq:massive_variance}), the asymptotics of (\ref{eq:ratio_Z}) for small $m$ gives the
desired upper-bound.
\end{proof}

\noindent The idea is to tilt the measure, as in the proof for $d\geq3$, first to work with the
massive measure, and second to follow the environment such that the field $\varphi$ is
shifted up of an amount $s$ on the sites $x$ for which $b\cdot e_x+h<0$. For some technical
reason, we need to work with the measure with boundary condition $a$,
so we perform three changes of measure (first one for changing boundary
condition, a second one for adding mass, and a third one for following
the environment).\\
Let $s>0$ and $m>0$ to be fixed later.
\begin{multline*}
f^\bfq_{{\Lambda_n}}(\ee)  
=
 n^{-2}\log
 \mu^{0,0}_{\Lambda_n}\left(\exp\sum_{x\in\Lambda_n}(b\cdot e_x+h)\ind{\phi_x\in[-a,a]}\right)\\
= n^{-2}\log
\mum^{0,\zeta}_{{\Lambda_n},m,\ee,s}\left(\exp\left(\sum_{x\in\Lambda_n}(b\cdot
    e_x+h)\ind{\phi_x\in[-a,a]}\right.\right.\\
\left.\left.
+\log\frac{\mathrm
  d\mu^{0,0}_{{\Lambda_n}}}{\mathrm d\mu^{0,\zeta}_{{\Lambda_n}}}
+\log\frac{\mathrm
d\mu^{0,\zeta}_{{\Lambda_n}}}{\mathrm d\mum^{0,\zeta}_{{\Lambda_n},m}}
+\log\frac{\mathrm
d\mum^{0,\zeta}_{{\Lambda_n},m}}{\mathrm d\mum^{0,\zeta}_{{\Lambda_n},m,\ee,s}}\right)\right)
\end{multline*}
where $\left({\varphi_{x}}\right)_{x\in\Lambda_n}$ under $\mum^{0,\zeta}_{{\Lambda_n},m,\ee,s}$
is distributed as $\left({\varphi_{x}+s\ind{{(b\cdot e_x+h)<0}}}\right)_{x\in\Lambda_n}$
under $\mum^{0,\zeta}_{{\Lambda_n},m}$ ; more formally, introducing 
$T_{\ee,s}:
(\left({\varphi_{x}}\right)_{x\in\Lambda_n})\mapsto\left({\varphi_{x}+s\ind{{(b\cdot e_x+h)<0}}}\right)_{x\in\Lambda_n}$,
we define $\mum^{0,\zeta}_{{\Lambda_n},m,\ee,s}$ as $\mum^{0,\zeta}_{{\Lambda_n},m}\circ T_{\ee,s}^{-1}$.
Using Jensen's inequality we get
\begin{eqnarray*}
f^\bfq_{{\Lambda_n}}(\ee)&\geq& n^{-2}
\mum^{0,\zeta}_{{\Lambda_n},m,\ee,s}\left(\sum_{x\in\Lambda_n}(b\cdot e_x+h)\ind{\varphi_x\in[-a,a]}
+\underbrace{\log\frac{\mathrm
  d\mu^{0,0}_{{\Lambda_n}}}{\mathrm d\mu^{0,\zeta}_{{\Lambda_n}}}}_{(1)}
+\underbrace{\log\frac{\mathrm
d\mu^{0,\zeta}_{{\Lambda_n}}}{\mathrm d\mum^{0,\zeta}_{{\Lambda_n},m}}}_{(2)}
+\underbrace{\log\frac{\mathrm
d\mum^{0,\zeta}_{{\Lambda_n},m}}{\mathrm d\mum^{0,\zeta}_{{\Lambda_n},m,\ee,s}}}_{(3)}\right)
\end{eqnarray*}
As in the proof for $d\geq3$, we have
$$n^{-2}\mum^{0,\zeta}_{{\Lambda_n},m,\ee,s}((1))\geq-Cn^{-1}.$$
By Lemma \ref{lem:massive}, we have
$\frac{\Zm^{0,\zeta}_{{\Lambda_n},m}}{Z^{0,\zeta}_{{\Lambda_n}}}=\frac{\Zm^{0,\zeta}_{{\Lambda_n},m}}{\Zm^{0,0}_{{\Lambda_n},m}}\frac{\Zm^{0,0}_{{\Lambda_n},m}}{Z^{0,0}_{{\Lambda_n}}}\frac{Z^{0,0}_{{\Lambda_n}}}{Z^{0,\zeta}_{{\Lambda_n}}}\geq
-Cn-C_2n^2m^2|\log m|$,
and then
\[
(2)
=\log\left(\frac{\Zm^{0,\zeta}_{{\Lambda_n},m}}{Z^{0,\zeta}_{\Lambda_n}}\right)+m^2\sum_{x\in\Lambda_n}\varphi_x^2
\geq-Cn -C_2n^2m^2|\log m|+ m^2\sum_{x\in\Lambda_n}\varphi_x^2
\]
hence,
\[
n^{-2}\mum^{0,\zeta}_{{\Lambda_n},m,\ee,s}((2))\geq-Cn^{-1} -C_2m^2|\log m|.
\]
Finally, noticing that $\Zm^{0,\zeta}_{{\Lambda_n},m,\ee,s}=\Zm^{0,\zeta}_{{\Lambda_n},m}$ (just perform a change
of variables in the Gaussian integral), we can compute the third
term.
\[
(3)=-\frac1{8}\sum_{\substack{\{x,y\}\cap\Lambda_n\neq\varnothing\\x\sim
   y}}(\varphi_x-\varphi_y)^2-(\hat\varphi_x-\hat\varphi_y)^2-m^2\sum_{x\in\Lambda_n}(\varphi_x-s)^2-(\hat\varphi_x-s)^2
\]
where $\hat\varphi_x:=\varphi_x+s\ind{b\cdot e_x+h<0}$. Now we will use the fact that the marginal laws of all $\varphi_x$, $x\in\Lambda_n$ under
$\mum^{0,\zeta}_{{\Lambda_n},m}$ are Gaussian variables,
i.e.\ $\varphi_{x}\sim\mathcal{N}(\mu_n^x,{\sigma_n^x}^2)$ where
$\mu_n^x\approx \zeta$ except for $x$ close to the boundary of the
box. Indeed, by the random walk representation of the mean, there is $C>0$ such that 
  $|\mum^{0,\zeta}_{\Lambda_n,m}(\varphi_x)-\zeta|\leq C(1+m^2)^{-\Vert x-\partial\Lambda_n\Vert}$. Moreover, ${(\sigma_n^x)}^2=\mbox{\textbf{V}ar}^{0,\zeta}_{{\Lambda_n},m}(\varphi_x)\leq C_1|\log
m|$. Using the
definition of $\mum^{0,\zeta}_{{\Lambda_n},m,\ee,s}$, and computing the terms as in the proof for $d\geq3$,
\begin{eqnarray*}
n^{-2} \mum^{0,\zeta}_{{\Lambda_n},m,\ee,s}((3))
&\geq& -\frac{s^2}{8n^{2}} \sum_{\substack{\{x,y\}\cap\Lambda_n\neq\varnothing\\x\sim
  y}} (\ind{b\cdot e_x+h<0} -\ind{b\cdot e_y+h<0})^2 -\frac{ m^2
s^2}{n^{2}} \sum_{x\in \Lambda_n} \ind{b\cdot e_x+h<0} +\frac Cn.
\end{eqnarray*}
We get, for $n$ large enough and $m$ small enough
\begin{eqnarray*}
f_{\Lambda_n}^\bfq(\ee)&\geq&
n^{-2}\sum_{x\in\Lambda_n}(b\cdot e_x+h)\mum^{0,\zeta}_{{\Lambda_n},m}(\hat\varphi_x\in[-a,a])
-\frac{s^2}{8n^{2}} \sum_{\substack{\{x,y\}\cap\Lambda_n\neq\varnothing\\x\sim
  y}} (\ind{b\cdot e_x+h<0}-\ind{b\cdot e_y+h<0})^2\\
&&- \frac{ m^2 s^2}{n^2} \sum_{x\in \Lambda_n} \ind{b\cdot e_x+h<0} -C'm^2|\log m|-Cn^{-1},
\end{eqnarray*}  
for some $C$ and $C'>0$.
Note that for $O(n^2)$ sites $x$, we have
\begin{eqnarray}\label{overlap_gaussian_d2}
\mum^{0,\zeta}_{{\Lambda_n},m}({\varphi_{x}\in[-a,a]})-\mum^{0,\zeta}_{{\Lambda_n},m}({\varphi_{x}+s\in[-a,a]})
&\asymp&\Phi'_{\zeta,b}(a)\cdot s\quad \mbox{ as } n\to\infty,
\end{eqnarray}
for $s<<a\leq \zeta=C_1|\log m|$, and $b=C_1|\log m|$. Above $\Phi_{\zeta,b}$ stands for
the p.d.f. of the above Gaussian distribution with mean $\zeta$ and
variance $b^2$. In particular, for a positive fraction of $x$
(close to 1) and $m$ sufficiently small, we have the upper bound :
\begin{eqnarray*}
\mum^{0,\zeta}_{{\Lambda_n},m}({\varphi_{x}\in[-a,a]})-\mum^{0,\zeta}_{{\Lambda_n},m}({\varphi_{x}+s\in[-a,a]})
&\geq&\frac{C_1(a)}{|\log m|}\cdot s,
\end{eqnarray*}
for some $C_1(a)>0$. Now we can compute :
\begin{eqnarray*}
f_{\Lambda_n}^\bfq(\ee)&\geq&
n^{-2}\sum_{x\in\Lambda_n}(b\cdot e_x+h) (\mum^{0,\zeta}_{{\Lambda_n},m}(\varphi_x\in[-a,a])-\frac{C_1(a)}{|\log
  m|}s\ind{b\cdot e_x+h<0})\\
&&-\frac{s^2}{8n^2} \sum_{\substack{\{x,y\}\cap\Lambda_n\neq\varnothing\\x\sim
  y}} (\ind{b\cdot e_x+h<0}-\ind{b\cdot e_y+h<0})^2- \frac{m^2 s^2}{n^2}
\sum_{x\in \Lambda_n} \ind{b\cdot e_x+h<0}\\
&&-C'm^2|\log m|-Cn^{-1}
\end{eqnarray*}
Observe that $\mum^{0,\zeta}_{{\Lambda_n},m}(\varphi_x\in[-a,a])\geq
2a\cdot\Phi'_{\zeta,b^2}(-a)=\frac{\tilde C_1(a)}{|\log m|}$ uniformly in $x\in\Lambda_n$.
Let us take the
expectation with respect to the environment, and use the bounded
convergence theorem and \cite[Theorem 2.1]{CoqMil2013_1}, we get :
{
\begin{eqnarray}\label{parabola_d2}
f^\bfq(\ee)=\lim_{n\to\infty}\E f^\bfq_{{\Lambda_n}}(\ee)&\geq&
h \frac{\tilde C_1(a)}{|\log m|}-s\cdot\frac{C_1(a) (-b+h)}{2|\log
  m|}-\frac{s^2 m^2}2 -\frac{ s^2}{16}-C'm^2|\log m|.
\end{eqnarray}
Our aim now is to show that the right hand side can be positive even
when $h$ is negative. In the above expression $s,m$ are free
parameters which we may vary. However, we have to remember that both
$s$ and $m$ need to be small enough, which makes standard optimization
analysis cumbersome. We are going to show that there exists $C>0$ and
$\epsilon>0$ such that for any $b,h$ such that $(-b+h) \in (-\epsilon,0)$ and
\[
	h := C \frac{(-b+h)^2}{\log (-(-b+h))},
\]
there exist small $s$ and $m$ such that the r.h.s.\ of \eqref{parabola_d2}
is positive. Notice that the result will imply that for any $h
\geq C \frac{(-b+h)^2}{\log (-(-b+h))}$ the free energy is positive. Let us choose the value of $s$ which
maximizes \eqref{parabola_d2} for fixed $m$, i.e.\
\begin{equation}
	s =  - \frac{C_1(a) (-b+h)}{ (m^2 + 1/4) |\log(m)|}. \label{eq:vache1}
\end{equation}
and for $m$ let us take
\begin{equation}
	m^2 = - k/(\log k)^3,  \text{ where } k := -h \tilde{C}_1(a)/C'. \label{eq:vache2}
\end{equation}
One can verify that with the above choice of parameters both $s$ and
$m$ are as small as we want. Let us first put \eqref{eq:vache1} into
the r.h.s.\ of \eqref{parabola_d2} and obtain
\begin{eqnarray*}
f^\bfq(\ee)&\geq&
h \frac{\tilde C_1(a)}{|\log m|}  + \frac{C_1(a)^2  (-b+h)^2}{(m^2+ 1/4)(\log m)^2}-C'm^2|\log m|.
\end{eqnarray*}
For $k$ and consequently $m$ small enough we have
\begin{eqnarray*}
f^\bfq(\ee)&\geq&
h \frac{\tilde C_1(a)}{|\log m|}  + \frac{C_1(a)^2  (-b+h)^2}{2(\log m)^2}-C'm^2|\log m|.
\end{eqnarray*}
Further let us multiply both sides by $(\log m)^2$ and insert \eqref{eq:vache2}.
\begin{eqnarray*}
	f^\bfq(\ee) (\log m)^2 &\geq& -h \tilde C_1(a) \log m  +
        \frac{C_1(a)^2 (-b+h)^2}{2}+C'm^2(\log m)^3 \\
	&=& -\frac{h \tilde C_1(a)}2 (\log k - 3 \log (|\log k|))  +
        \frac{C_1(a)^2  (-b+h)^2}{2}-C' \frac {k (\log k - 3 \log (|\log k|))^3}{8(\log k)^3} \\
	&=&  -h\frac{ \tilde C_1(a)}2 \log ( - h \tilde{C}_1(a)/C')  + \frac{C_1(a)^2
           (-b+h)^2}{2} +h \frac{\tilde{C}_1(a)}8+o(h)\quad \mbox{as} \quad |h|\to0
\end{eqnarray*}
From the last claim it is straightforward to conclude existence of $C$
(sufficiently small) and $\epsilon$ with the properties described above.\qed

}
\strut
 
\paragraph{Acknowledgements.} We would like to warmly acknowledge Yvan
Velenik for introducing us to the topic.
L.C.\ was partially supported by the Swiss National
Foundation. P.M.\ was supported by a Sciex Fellowship grant no.\
10.044.

\bibliographystyle{abbrv}
\bibliography{biblio_short}

\end{document}